\documentclass[a4paper,oneside,11pt]{article}
\usepackage[utf8]{inputenc}
\usepackage{lmodern,stackrel}
\usepackage[T1]{fontenc}
\usepackage{verbatim}
\usepackage{textcomp}
\usepackage[english]{babel}
\usepackage[pdftex]{hyperref}
\usepackage[pdftex]{color,graphicx}
\usepackage{amssymb}
\usepackage{amsmath}
\usepackage{upgreek}
\usepackage{graphicx, epstopdf}
\usepackage{dsfont}

\renewcommand{\leq}{\leqslant}
\renewcommand{\geq}{\geqslant}

\usepackage{amsmath,amsfonts,amssymb,amsthm,mathrsfs,stackrel}

\def\build#1_#2^#3{\mathrel{
\mathop{\kern 0pt#1}\limits_{#2}^{#3}}}

\theoremstyle{plain}
\newtheorem{theorem}{Theorem}

\theoremstyle{definition}

\theoremstyle{remark}
\newtheorem*{remark}{Remark}

\newcommand{\Z}{{\mathbb{Z}}}

\begin{document}
\title{A note on truncated long-range percolation with heavy tails on oriented graphs}
\author{C.T.M. Alves\footnote{Departamento de Estat\'istica, IMECC, Universidade Estadual de Campinas,  rua S\'ergio Buarque de Holanda 651,
13083--859, Campinas SP, Brazil}, M. Hil\'ario\footnote{Departamento de Matem{\'a}tica, Universidade Federal de Minas Gerais, Av. Ant\^onio
Carlos 6627 C.P. 702 CEP 30123-970 Belo Horizonte-MG, Brazil}, B.N.B. de Lima$^\dagger$, D. Valesin\footnote{Johann Bernoulli Instituut, Rijksuniversiteit Groningen, Nijenborgh 9 9747 AG Groningen, The Netherlands}}
\date{}
\maketitle


\begin{abstract}
We consider oriented long-range percolation on a graph with vertex set $\Z^d \times \Z_+$ and directed edges of the form $\langle (x,t), (x+y,t+1)\rangle$, for $x,y$ in $\Z^d$ and $t \in \Z_+$. Any edge of this form is open with probability $p_y$, independently for all edges. Under the assumption that the values $p_y$ do not vanish at infinity, we show that there is percolation even if all edges of length more than $k$ are deleted, for $k$ large enough. We also state the analogous result for a long-range contact process on $\Z^d$.
\end{abstract}
{\footnotesize Keywords: contact processes; oriented percolation; long-range percolation; truncation }

\section{Introduction}

\noindent
Let $G = (\mathbb{V}, \mathbb{E})$ be the graph with set of vertices $\mathbb{V} = \Z^d \times \Z_+$ and  set of (oriented) bonds 
\begin{equation}\label{eq:def_edges_G}\mathbb{E} =\left\{ \langle (x,t), (x+y,t+1)\rangle :\; x,y \in \Z^d,\; t \in \Z_+ \right\}.\end{equation}
Let $(p_y)_{y \in \Z^d}$ be a family of numbers in the interval $[0,1]$  and consider a Bernoulli bond percolation model where each bond $\langle (x,t), (x+y,t+1)\rangle\in \mathbb{E}$ is open with probability $p_y$, independently for all bonds.  That is, take $(\Omega, \, \mathcal{A}, \, P)$, where $\Omega = \{0,1\}^{\mathbb{E}}$, $\mathcal{A}$ is the canonical product $\sigma$-algebra, and $P = \prod_{e \in \mathbb{E}} \mu_e$, where $\mu_e({\omega}_e = 1) = p_{y} = 1- \mu_e({\omega}_e = 0)$ for $e = \langle (x,t), (x+y,t+1)\rangle \in\mathbb{E}$. An element $\omega \in \Omega$ is called a percolation configuration.

A (finite or infinite) sequence $(v_0, v_1, \dots )$ with $v_i \in G$ for each $i$ is called an oriented path if, for each $i$, $v_i-v_{i-1}=(y,1)$ for some $y\in\mathbb{Z}^d$; the oriented path is open if each oriented edge $\langle v_i, v_{i+1} \rangle$ is open.
For $(x,t), (x',t') \in \mathbb{V}$ with $t < t'$, we denote by $\{(x,t) \rightsquigarrow (x',t')\}$ the event that there is an open oriented path from $(x,t)$ to $(x',t')$. If $A \subset \mathbb{V}$, we denote by $\{(x,t) \rightsquigarrow A\}$ the event that $(x,t)$ is connected by an open oriented path to some vertex of $A$. Finally,  we denote by $\{(x,t) \rightsquigarrow \infty\}$ the event that there is an infinite open oriented path started from $(x,t)$.

We now consider a truncation of the family $(p_y)_{y \in \Z^d}$ at some finite range $k$. More precisely, for each $k\in\mathbb{N}$ consider the truncated family $(p_y^k)_{y \in \Z^d}$, defined by
\begin{equation}
p_y^k=\left\{
\begin{array}
[c]{l}%
p_y,\mbox{  if } \|y\|_\infty \leq k,\\
0,\ \mbox{   otherwise},
\end{array}\right.\label{eq:truncation}
\end{equation}
and the measure $P^k = \prod_{e \in \mathbb{E}} \mu^k_e$, where $\mu^k_e({\omega}_e = 1) = p^k_{y} = 1- \mu^k_e({\omega}_e = 0)$ for $e = \langle (x,t), (x+y,t+1)\rangle \in\mathbb{E}$. Then, one can ask the \textbf{truncation question}: is it the case that, whenever percolation can occur for a sequence of connection probabilities, it can also occur for a sufficiently high truncation of the sequence? That is: in case $P\{0 \rightsquigarrow \infty\} > 0$, is there a large enough truncation constant $k$ for which we still have $P^k( 0 \rightsquigarrow\infty) >0$ ?

Numerous works (\cite{MS, Be, SSV, MSV, FLS, FL, LS, ELV} in chronological order) addressed this question considering different models (such as: the Ising model, oriented and non-oriented percolation, the contact process) or different assumptions on the sequence $(p_n)$ or on the graph. We direct the reader to the introductory sections of \cite{FL} and \cite{ELV} for a more thorough discussion. Our main contribution is the following:

\begin{theorem}\label{thm:main_cor} If there exists $\varepsilon > 0$ such that $p_y > \varepsilon$ for infinitely many vectors $y$, then the truncation question has an affirmative answer. Moreover, $${\displaystyle \lim_{k\to\infty}P^k\{(0,0) \rightsquigarrow \infty \} = 1}.$$
\end{theorem}

This result generalizes the analogous result obtained in \cite{FLS} for non-oriented percolation on the square lattice. In that paper, the authors were able to construct a proper subgraph of $\Z^2$ with long (but limited) range edges that was isomorphic to a slab with two ``unbounded'' directions and arbitrarily large number of ``bounded'' dimensions and thickness. This allowed them to apply~\cite{GM} to obtain their result. In our case however, this approach is fruitless, since~\cite{GM} is not applicable in the case of oriented percolation processes. Therefore, we need to devise a new strategy.

In Section \ref{sec:variations}, we present two settings where a positive answer to the truncation question can be readily obtained from the above theorem: an anisotropic two-dimensional oriented percolation model and a long-range contact process on $\Z^d$. We prove Theorem \ref{thm:main_cor} in Section \ref{sec:proof}.

\section{Proof of Theorem \ref{thm:main_cor}}
\label{sec:proof}
We first prove the theorem for the case where $d = 1$, so that the family $(p_y)$ is given by a doubly-infinite sequence $(\ldots, p_{-1}, p_0,p_1,\ldots)$ (we replace $y$ by $n$ in the notation). Moreover, we assume that $p_n = 0$ if $n \leq 0$. In the end of this section, we will show how  we can obtain the general statement from this particular case.

By assumption, we can take $\epsilon>0$ such that $\limsup_{n\to\infty}  p_n>\epsilon>0$. Define the sequence $(a_n)_n$ as $$a_1=\inf\{i:\; p_i>\epsilon\},\qquad a_n=\inf\{i>a_{n-1}:\; p_i>\epsilon\},\; n >1.$$ Fix $\delta\in (0,1)$ to be chosen later. Define the integers $L_0$ and $L_1$
\begin{equation}
\label{eq:bin_cond} 
P(\text{Bin}(L_0,\epsilon)\geq 1)>1-\frac{\delta}{3},\quad P(\text{Bin}(L_1,\epsilon)\geq L_0)>1-\frac{\delta}{3}
\end{equation}
(here $\text{Bin}(n,p)$ denotes a Binomial distribution with parameters $n$ and $p$). Next, define $R$ such that
\begin{equation}\label{eq:defR}R=\max\{a_{L_1}, a_{2L_1}-a_{L_1}\}.\end{equation} Finally, take $L_2$ large enough such that 
\begin{equation}\label{l2}
a_{L_2}>a_1+3R. 
\end{equation}

Given a vertex $(x,y)\in\mathbb{Z}^2_+$ and $i\in\mathbb{N}$, define the events
\begin{align*}&R_i^{(x,y)}=\left\{\begin{array}{l}\langle (x,y),(x+i,y+1)\rangle\mbox{ and }\\\langle (x+i,y+1),(x+i+a_1,y+2)\rangle\mbox{ are open}\end{array}\right\},\\[.2cm] &S_i^{(x,y)}=\left\{\begin{array}{l}\langle (x,y),(x+i,y+1)\rangle\mbox{ and }\\\langle (x+i,y+1),(x+i+a_{L_2},y+2)\rangle\mbox{ are open}\end{array}\right\}.\end{align*} Also define \begin{align*}&T_-^{(x,y)}=\left(\cup_{i=1}^{a_{L_1}}R_i^{(x,y)}\right)\cap\left(\cup_{i=1}^{a_{L_1}}S_i^{(x,y)}\right),\\[.2cm]&T_+^{(x,y)}=\left(\cup_{i=a_{L_1}+1}^{a_{2L_1}}R_i^{(x,y)}\right)\cap\left(\cup_{i=a_{L_1}+1}^{a_{2L_1}}S_i^{(x,y)}\right).\end{align*}
\begin{figure}[t]
\centering
\fbox{\includegraphics[width=\textwidth]{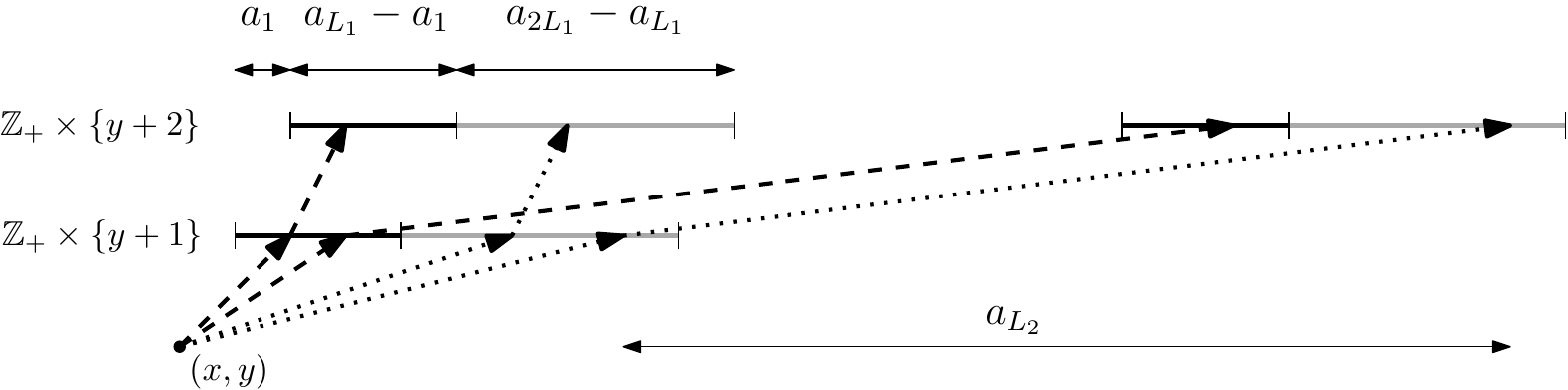}}
\caption{A realization of the events $T^{(x,y)}_-$ and~$T^{(x,y)}_+$, with dashed (respectively, dotted) lines representing open directed edges.}
\label{fig:seedevents}
\end{figure}
Observe that by \eqref{eq:bin_cond}, 
\begin{equation}
\label{eq:tmaismenos}
P^{a_{L_2}}\left( T_-^{(x,y)}\right)>1-\delta,\qquad P^{a_{L_2}}\left( T_+^{(x,y)}\right)>1-\delta.
\end{equation}
Also,
\begin{equation}\label{eq:propt-}
\begin{split}
\text{on } T^{(x,y)}_-, \;&(x,y) \rightsquigarrow [x+2a_1,\;x+a_1 +a_{L_1}] \times \{y+2\},\\  &(x,y) \rightsquigarrow [x+a_1+ a_{L_2},\;x+a_{L_1} +a_{L_2}] \times \{y+2\}
\end{split}
\end{equation}
and
\begin{equation}\label{eq:propt+}
\begin{split}
\text{on } T^{(x,y)}_+,\;&(x,y) \rightsquigarrow [x+a_1 + a_{L_1},\;x+a_1 +a_{2L_1}] \times \{y+2\},\\ \;&(x,y) \rightsquigarrow [x+a_{L_2} + a_{L_1},\;x+a_{L_2} +a_{2L_1}] \times \{y+2\}\end{split}
\end{equation}
(note that, by \eqref{eq:defR} and \eqref{l2}, the two horizontal segments in \eqref{eq:propt+} are disjoint, and similarly in \eqref{eq:propt-}).

The next step   is to define a renormalized lattice $G^*$ (also an oriented graph); vertices of $G^*$ will correspond to certain horizontal line segments in the original graph $G$. An exploration of the points reachable from the origin in $G$ under the measure $P^{a_{L_2}}$ will produce, as its `coarse-grained' counterpart, a site percolation configuration on $G^*$. As is usual, two properties will result from the coupling: first, percolation in $G^*$ will occur with high probability, and second, percolation in $G^*$ will imply percolation in $G$.

We let $G^*=(\mathbb{V}^*,\mathbb{E}^*)$, where $\mathbb{V}^*=\{(i,j)\in\mathbb{Z}\times\mathbb{Z}_+; i+j\mbox{ is even}\}$ and  $\mathbb{E}^*$ is the set of oriented edges $\mathbb{E}^*=\{\langle(i,j),(i\pm 1,j+1)\rangle;(i,j)\in\mathbb{V}^*\}$. Define the following order in $\mathbb{V}^*$: given $(i_1,j_1),(i_2,j_2)\in\mathbb{V}^*$ we say that $(i_1,j_1)\prec(i_2,j_2)$ if and only if $j_1<j_2$ or $(j_1=j_2\mbox{ and }i_1<i_2)$.
Given $S\subset\mathbb{Z} \times \mathbb{Z}_+$, we define the exterior boundary of $S$ as the set $$\partial_e S=\{(i,j)\in \mathbb{V}^*\backslash S; (i-1,j-1)\in S\mbox{ or }(i+1,j-1)\in S\}.$$

For each $(i,j) \in \mathbb{V}^*$, define 
$$z_{i,j} = j\cdot a_{L_1} + \frac{i+j}{2} \cdot a_{L_2} + \frac{j-i}{2}\cdot  a_{1}.$$
Also let
$$v_{i,j} = (z_{i,j}, 2j) \in \mathbb{V},\qquad I_{i,j} = [z_{i,j}-R,\;z_{i,j}+R] \times \{2j\} \subset \mathbb{V}.$$
These vertices and intervals are depicted in Figure \ref{fig:renormgraph}. Note that, for all $(i,j)$,
\begin{align}\label{eq:dist_hor}
&z_{i+2,j} - z_{i,j} = a_{L_2} - a_1,
\end{align}
so, by the choice of $L_2$ in \eqref{l2}, the segments $I_{i,j}$ are pairwise disjoint. Additionally,
\begin{align}\label{eq:dist_ver1}
&z_{i-1,j+1} - z_{i,j} = a_1 + a_{L_1},\\
&z_{i+1,j+1} - z_{i,j} = a_{L_1} + a_{L_2}.\label{eq:dist_ver2}
\end{align}

\begin{figure}[t]
\centering
\fbox{\includegraphics[width=\textwidth]{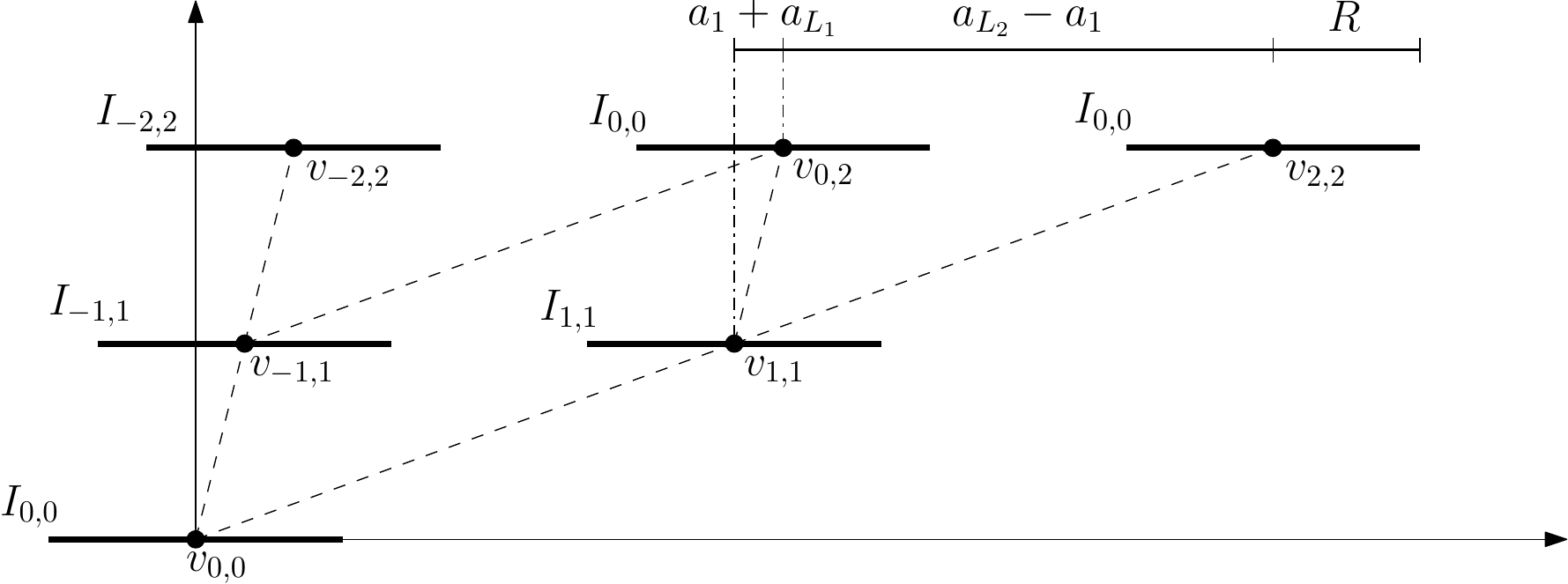}}
\caption{The vertices $v_{i,j}$ and the line segments $I_{i,j}$, for $(i,j) \in \mathbb{V}^*$.}
\label{fig:renormgraph}
\end{figure}

Let us now present our exploration algorithm. We will define inductively two increasing sequences $(A_i)_i$ and $(B_i)_i$ of subsets of $\mathbb{V}^*$. Set $A_0=B_0=\emptyset$ and $x_0=(0,0)$. We declare the vertex $x_0=(0,0)$ as good if the event $T_-^{(0,0)}$ occurs. Then, we define:
\begin{align*}
&A_{1} =
\begin{cases}
A_0\cup\{x_0\},&\mbox{  if } x_0\mbox{ is good},\\
A_0, &\mbox{   otherwise},
\end{cases}\qquad B_{1}=\begin{cases}
B_0,&\mbox{  if } x_0\mbox{ is good},\\
B_0\cup\{x_0\},& \mbox{  otherwise.}
\end{cases}
\end{align*}
If $x_0$ is not good, then we stop our recursive procedure. Note that, if $x_0$ is good, then by \eqref{eq:propt+} and \eqref{eq:propt-},
\begin{align*}
&(0,0) \rightsquigarrow [2a_1,\;a_1 + a_{L_1}] \times \{2\}\\&\hspace{2cm} \subset [a_1 + a_{L_1} - R,\; a_1 + a_{L_1} + R] \times \{2\} = I_{-1,1},\\
&(0,0) \rightsquigarrow [a_1 + a_{L_2},\;a_{L_1} + a_{L_2}] \times \{2\}\\&\hspace{2cm}  \subset [a_{L_1} + a_{L_2} - R,\; a_{L_1} + a_{L_2} + R] \times \{2\} = I_{1,1}.
\end{align*}

Assume $A_n, B_n$ have been defined for $n \geq 1$, and the following conditions are satisfied:
\begin{enumerate}
\item[(a)] $A_n$ is connected,
\item[(b)] $B_n \subset \partial_e A_n$, 
\item[(c)] bonds started from vertices outside $\cup_{(i,j) \in A_n \cup B_n} I_{i,j}$ are still unexplored, and
\item[(d)] $(0,0) \rightsquigarrow I_{i,j}$ for each $(i,j) \in (\partial_e A_n) \backslash B_n$.
\end{enumerate}
Now, if $(\partial_e A_{n})\backslash B_{n} = \emptyset$ we stop our recursive definition. Otherwise we let $x_{n}= (i,j)$ be the minimal point of $(\partial_e A_{n})\backslash B_{n}$ with respect to the order $\prec$ defined above. By property (d) above, we can fix a vertex $(u,2j) \in I_{i,j}$ such that $(0,0) \rightsquigarrow (u,2j)$. In case  $u \in [z_{i,j}-R, z_{i,j}] $
that is, $(u,2j)$ belongs to the left half of $I_{i,j}$ (including the midpoint), then we declare that $x_n$ is good if the event $T^{(u,2j)}_+$ occurs. In case $(u,2j) \in (z_{i,j},z_{i,j}+R]$, then we declare the $x_n$ is good if the event $T^{(u,2j)}_-$ occurs.  Then we define
\begin{align*}
&A_{n+1} =
\begin{cases}
A_n\cup\{x_n\},&\mbox{  if } x_n\mbox{ is good},\\
A_n, &\mbox{   otherwise},
\end{cases}\quad B_{n+1}=\begin{cases}
B_n,&\mbox{  if } x_n\mbox{ is good},\\
B_n\cup\{x_n\},& \mbox{  otherwise.}
\end{cases}
\end{align*}
It is clear that (a), (b), (c) listed above are satisfied with $A_{n+1}, B_{n+1}$ in the place of $A_n, B_n$. Let us now verify that our steering mechanism (that is, choosing $T_+$ or $T_-$ according to the position of $(u,2j)$) guarantees property (d). Consider first the case where $(u,2j)$ is in the left half of $I_{i,j}$, that is, $u \in [z_{i,j}-R,z_{i,j}]$; then,
\begin{align*}
&u+a_1+a_{L_1} \geq z_{i,j} - R +a_1 + a_{L_1} \stackrel{\eqref{eq:dist_ver1}}{=} z_{i-1,j+1} -R,\\
&u+a_1+a_{2L_1} \leq z_{i,j} + a_1 + a_{2L_1} \stackrel{\eqref{eq:dist_ver1}}{=} z_{i-1,j+1} + a_{2L_1}-a_{L_1} \stackrel{\eqref{eq:defR}}{\leq} z_{i-1,j+1} + R,\\[.1cm]
&u+a_{L_2} + a_{L_1} \geq z_{i,j} - R + a_{L_2} + a_{L_1} \stackrel{\eqref{eq:dist_ver2}}{=} z_{i+1,j+1} - R,\\[.1cm]
&u+a_{L_2} + a_{2L_1} \leq z_{i,j} + a_{L_2} + a_{2L_1} \stackrel{\eqref{eq:dist_ver2}}{=} z_{i+1,j+1} + a_{2L_1} - a_{L_1} \stackrel{\eqref{eq:defR}}{=} z_{i+1,j+1}+R,
\end{align*}
so \eqref{eq:propt+} implies that, if $T^{(u,2j)}_+$ occurs, we have $$(0,0) \rightsquigarrow (u,2j) \rightsquigarrow I_{i-1,j+1},\qquad 
(0,0)\rightsquigarrow (u,2j) \rightsquigarrow I_{i+1,j+1}.$$ The case where $(u,2j)$ is in the right half of $I_{i,j}$ is treated similarly (using \eqref{eq:propt-}). This completes the proof that (d) remains satisfied after each recursion step.

Regardless of whether or not the recursion ever ends, we let    ${\cal C}$ be the union of all sets $A_n$ that have been defined. By construction, it follows that $\{|{\cal C}| =\infty\} \subseteq \{(0,0) \rightsquigarrow \infty\}$.

Now, observe that
\begin{equation} \label{eq:for_comp}P^{a_{L_2}}(x_{n}\mbox{ is good} \mid (A_m,B_m): 0 \leq m \leq n)\geq 1-\delta.\end{equation}
This implies that ${\cal C}$ stochastically dominates the cluster of the origin in Bernoulli oriented site percolation on $G^*$ with parameter $1-\delta$ (see Lemma 1 of \cite{GM}). As $\delta$ can be taken arbitrarily small, this proves the desired result for $d = 1$.

Now let us show how the statement of Theorem \ref{thm:main_cor} can be obtained from the case we have already treated. Take $\epsilon > 0$ as in the assumption of the theorem; we can then take an infinite set $S \subset \Z^d$ so that $p_y > \epsilon$ for all $y \in S$.

Let 
\begin{align*}&\Pi^-_i(S) = \left\{\begin{array}{ll} x_i < 0: &(y_1,\ldots, y_{i-1}, x_i,y_{i+1},\ldots, y_d) \in S \\&\text{ for some } y_1,\ldots, y_{i-1},y_{i+1},\ldots, y_d \in \Z\end{array}\right\},\\
&\Pi^+_i(S) = \left\{\begin{array}{ll} x_i > 0: &(y_1,\ldots, y_{i-1}, x_i,y_{i+1},\ldots, y_d) \in S \\&\text{ for some } y_1,\ldots, y_{i-1},y_{i+1},\ldots, y_d \in \Z\end{array}\right\}
\end{align*}
(in words, these sets are given by the projection of $S$ to the $i$th axis, intersected with $(-\infty,0)$ and $(0,\infty)$, respectively). Since $S$ is infinite, there exists $i \in \{1,\ldots, d\}$ and $a\in \{-,+\}$ such that $\Pi^a_i(S)$ is infinite; for simplicity, assume that this is the case for $a = +$ and $i = 1$. It is then easy to see that the cluster of 0 for percolation on $G$, when projected on the first coordinate axis times~$\Z_+$, stochastically dominates a percolation configuration on $\Z\times \Z_+$ which belongs to the case we have already treated. Percolation of this configuration then implies percolation on $G$.

\section{Truncation question for related oriented models}
\label{sec:variations}
In this section, we consider different oriented percolation models in which the truncation question can be posed, and an affirmative answer follows almost directly from Theorem \ref{thm:main_cor}.

\subsection{Anisotropic oriented percolation on the square lattice}
\label{ss:aniso}

For the first model, let ${\cal G}=(\mathbb{Z}^2,{\cal E})$, where ${\cal E} = {\cal E}_v \cup (\cup_{n=1}^\infty {\cal E}_{h,n})$: \begin{align*}&{\cal E}_v=\{\langle(x,y),(x,y+1)\rangle: x,y\in \mathbb{Z}_+\},\\&{\cal E}_{h,n} =\{\langle(x,y),(x+n,y)\rangle: x,y\in \mathbb{Z}_+,n\in\mathbb{N}\},\end{align*} that is, ${\cal G}$ is an oriented square lattice equipped with long-range horizontal bonds. Given $\sigma > 0$ and $(q_n)_n$ with $q_n \in [0,1]$ for each $n$, we define an oriented bond percolation model where each bond $e$ is open, independently of each other,  with probability $\sigma$ or $q_{n}$, if $e\in{\cal E}_v$ or $e\in{\cal E}_{h,n}$, respectively. Let ${\cal P}$ be a probability measure under which this model is defined.

For the graph ${\cal G}$, an oriented path is a sequence $(v_1, v_2,\ldots)$ such that, for each $i$, $v_{i+1}-v_{i}=(0,1)$ or $(n,0)$ for some $n\in\mathbb{N}$; the path is open if each oriented bond in it is open. We use also the notation $\{(0,0) \rightsquigarrow \infty \}$ to denote the set of configurations such that the origin is connected to infinitely many vertices by oriented open paths on ${\cal  G}$.

As in Section 1, we denote by $(q_n^k)_n$ and ${\cal P}^k$ the truncated sequence and the truncated probability measure. Thus, for this graph we have a result analogous to Theorem \ref{thm:main_cor}:

\begin{theorem}\label{thm:z2} For the Bernoulli long-range oriented percolation model on ${\cal G}$, if $\limsup q_n>0$,  then the truncation problem has an affirmative answer. Moreover, $${\displaystyle \lim_{k\to\infty}{\cal P}^k\{(0,0) \rightsquigarrow \infty \} = 1}.$$
\end{theorem}

\begin{proof} From the percolation model on ${\cal G}$, we define an induced bond percolation model on the graph $G$ of the previous sections with $d = 1$, that is, $G = (\mathbb{V},\mathbb{E})$ with $\mathbb{V} = \Z \times \Z_+$ and $\mathbb{E}$ as in \eqref{eq:def_edges_G}. We declare each bond $\langle (x,y);(x+n,y+1)\rangle$ in $G$ as open if and only if both the  bonds $\langle (x,y);(x+n,y)\rangle$ and $\langle (x+n,y);(x+n,y+1)\rangle$ in ${\cal G}$ are open. Observe that:
\begin{itemize}
\item each bond $\langle (x,y);(x+n,y+1)\rangle \in \mathbb{E}$ is open with probability $p_n:=\sigma q_n$ and by hypothesis $\limsup q_n>0$;
\item if there is an infinite open path in the induced model on $G$ then this implies the existence of an infinite oriented path in the original model on ${\cal G}$;
\item the induced percolation model on $G$ is not an independent model, because the open or closed statuses for bonds with the same end vertex are positively correlated. However, given any collection of bonds in which any two bonds have distinct end vertices, the statuses of all these bonds are independent. Therefore, there exist no problems regarding the definition of events analogous to~$T^{(x,y)}_-$ and~$T^{(x,y)}_+$ and in showing lower bounds like~\eqref{eq:tmaismenos}.
\end{itemize} 
\begin{remark}
The range of the dependence on the induced model on $G$ goes to infinity as the truncation parameter $k\rightarrow\infty$. Hence, the conclusion of Theorem \ref{thm:z2} could not be derived from standard techniques of stochastic comparison with product measures (see for instance the main result in \cite{LSS}).
\end{remark}
We can now prove that, for the induced model on $G$, percolation occurs with high probability if $k$ is large by an argument that is almost identical to the one of the previous section. The only difference is that, in the renormalized site percolation configuration on $G^*$ that results from the exploration algorithm, some dependence with range one now arises. This is because the probability that a vertex $x_n = (i,j) \in \mathbb{V}^*$ in our exploration is good will be affected by a previous query of the vertex $(i-2,j)$. This issue is settled by choosing the constant $\delta$ so that, for one-dependent oriented percolation configurations on $G^*$ with density of open bonds above $1-\delta$, percolation occurs with high probability, since we are now in the context of one-dependent percolation, where~$\cite{LSS}$ applies.
\end{proof}

\subsection{Long-range contact processes on $\mathbb{Z}^d$}

The second model we consider is a contact process on $\mathbb{Z}^d$ with long-range interactions, such as the one considered in \cite{ELV}. To define the model, we fix a family of non-negative real numbers $(\lambda_y)_{y \in \Z^d}$, and take a family of independent Poisson point processes on $[0,\infty)$:
\begin{itemize}
\item a process $D^{x}$ of rate 1 for each $x \in \mathbb{Z}^d$;
\item a process $B^{(x,y)}$ of rate $\lambda_{x-y}$ for each ordered pair $(x, y)$ with $x,y\in \mathbb{Z}^d$.
\end{itemize}
We view each of these processes as a random discrete subset of $[0,\infty)$ and write, for $0\leq a < b$, $D^{x}_{[a,b]} = D^{x} \cap [a, b]$ and $B^{(x,y)}_{[a,b]} = B^{(x,y)} \cap [a,b]$. We let $\mathscr{P}$ be a probability measure under which these processes are defined.

Fix $k \in \mathbb{N}$. Given $x, y \in \mathbb{Z}^d$ and $0 \leq s \leq t$, we say $(x,s)$ and $(y,t)$ are $k$-connected, and write $(x,s) \stackrel{k}{\rightsquigarrow} (y,t)$, if there exists a function $\gamma:[s,t] \to \mathbb{Z}^d$ that is right-continuous, constant between jumps and satisfies:
\begin{equation}\label{eq:infect_path}
\begin{split}
&\gamma(s) =x,\; \gamma(t) = y \text{ and, for all }r \in [s,t],\\&\hspace{5cm} r \notin D^{\gamma(r)},\\&\hspace{5cm}r \in B^{(\gamma(r-),\gamma(r))} \text{ if } \gamma(r) \neq \gamma(r-),\\
&\hspace{5cm} \|\gamma(r) - \gamma(r-)\|_\infty \leq k .
\end{split}
\end{equation}
This provides a continuous-time percolation structure for the lattice $\mathbb{Z}^d$. From the point of view of interacting particle systems, one usually defines
$$\xi_{t,k}(x) = \mathds{1}\{(0,0) \stackrel{k}{\rightsquigarrow} (x,t)\},\qquad x \in \mathbb{Z}^d,\; t \geq 0,$$ where $\mathds{1}$ denotes the indicator function, thus obtaining a Markov process $(\xi_{t,k})_{t\geq 0}$ on the state space $\{0,1\}^{\Z^d}$. For this process, the identically  zero  configuration (denoted by $\underline{0}$) is absorbing.

\begin{theorem}
For the long-range contact process on $\mathbb{Z}^d$, if there exists $\underline{\lambda}  > 0$ such that $\lambda_y>\underline{\lambda} $ for infinitely many $y$, then $$ \lim_{k \to \infty} \mathscr{P}\left(\xi_{t,k} \neq \underline{0} \text{ for all } t \right) = 1.$$
\end{theorem}

\begin{proof}
Similarly to the proof in Section \ref{sec:proof}, we can easily reduce the proof to the case of $d= 1$ and  $\lambda_{a_n} > \underline{\lambda} > 0$ for an increasing sequence $(a_n)_{n\in\mathbb{N}}$. So we now turn to this case.

Fix $\delta > 0$. Choose $\tau > 0$ such that
\begin{equation}
\label{eq:choice_delta} \mathscr{P}(D^0_{[0,\tau]} \neq \varnothing) < \delta/4.
\end{equation} Given the Poisson processes $\{(D^x)_x, (B^{(x,y)})_{(x,y)}\}$, we now define a percolation configuration on the graph $G = (\mathbb{V},\mathbb{E})$. We declare a bond $\langle (x,n),(x+y,n+1)\rangle$ of $\mathbb{E}$ to be open if: $$D^x_{[\tau n, \tau(n+1)]} = \varnothing,\quad D^{x+y}_{[\tau n, \tau(n+1)]} = \varnothing,\quad B^{(x,x+y)}_{[\tau n, \tau(n+1)]} \neq \varnothing.$$
Let $P$ be the probability distribution of this induced percolation configuration, and $P^k$ the corresponding truncation (that is, the induced configuration obtained from $\{(D^x)_x, (B^{(x,y)})_{(x,y)}\}$ by suppressing the Poisson processes $B^{(x,y)}$ with $|y-x| > k$). 

We observe that
\begin{itemize}
\item each bond $\langle (x,n),(x+y,n+1)\rangle \in \mathbb{E}$ is open with probability larger than $$(1-\delta/4)^2 \cdot (1-\exp\{-\lambda_y \tau\});$$
\item if there is an infinite open path in the induced model on $G$ for some $k$, then we can construct a function $\gamma:[0,\infty) \to \mathbb{Z}$ with $\gamma(0) = 0$ and satisfying the three last requirements of \eqref{eq:infect_path} for $r \in [0,\infty)$, so that have $\xi_{k,t} \neq \underline{0}$ for all $t$;
\item given any collection of bonds in which any two bonds have distinct start vertices \textit{and} distinct end vertices, the statuses of all these bonds are independent. Note that there is more dependence here than in the model of Section \ref{ss:aniso} (since bonds with coinciding starting points are dependent here), so we have to be more careful in implementing the proof of Section \ref{sec:proof}.
\end{itemize}

We let $\epsilon = (1-\tfrac{\delta}{4})^2\cdot (1-\exp\{-\underline{\lambda} \tau\})$ and choose $L_0$ and $L_1$ such that
\begin{equation}\label{eq:bins_contact}
P(\text{Bin}(L_0,\epsilon) > 0) > 1-\frac{\delta}{8},\qquad P(\text{Bin}(L_1,\epsilon) > L_0) > 1-\frac{\delta}{8}.
\end{equation}
We now choose $R$ and $L_2$ and, for $(x,y) \in \Z^2_+$, we define events $R^{(x,y)}_i$, $S^{(x,y)}_i$, $T^{(x,y)}_-$ and $T^{(x,y)}_+$ exactly as in Section \ref{sec:proof}. 

Note that the event $\cup_{i=1}^{a_{L_1}} R^{(0,0)}_i$ is guaranteed to occur if the following items are satisfied:
\begin{itemize}
\item[(a)] $D^0_{[0,\tau]} = \varnothing$;
\item[(b)] for at least $L_0$ indices $i \in \{1,\ldots, L_1\}$, we have
$$D^{a_i}_{[0,\tau]} = \varnothing,\quad B^{(0,a_i)}_{[0,\tau]} \neq \varnothing. $$
\item[(c)] out of the indices $i$ satisfying the requirements of item (b), at least one also satisfies
$$D^{a_i}_{[\tau,2\tau]} = \varnothing,\qquad D^{a_i + a_1}_{[\tau,2\tau]} = \varnothing,\qquad B^{(a_i,a_i + a_1)}_{[\tau,2\tau]} \neq \varnothing. $$
\end{itemize}
Hence, by \eqref{eq:choice_delta} and \eqref{eq:bins_contact}, we have $P(\cup_{i=1}^{a_{L_1}} R^{(0,0)}_i) > 1-\delta/2$, and, by translation invariance of the Poisson processes, $P(\cup_{i=1}^{a_{L_1}} R^{(x,y)}_i) > 1-\delta/2$ for any $(x,y) \in \Z^2_+$. Similarly, we have $P(\cup_{i=1}^{a_{L_1}} S^{(x,y)}_i) > 1-\delta/2$, so that $$P\left(T^{(x,y)}_- \right) > 1-\delta,$$
and the same argument shows that
$$P\left(T^{(x,y)}_+ \right) > 1-\delta$$
also holds.

From here onward, the proof proceeds as in Section \ref{sec:proof}, with the only difference that already appeared in the treatment of the model of Section \ref{ss:aniso}: in the site percolation configuration  in the lattice $G^*$ that results from the exploration algorithm, dependence of range one arises. As in Section \ref{ss:aniso}, this issue is resolved (and percolation is guaranteed) as soon as $1-\delta$ is supercritical for one-dependent site percolation on $G^*$.
\end{proof}

\section*{Acknowledgements}
The authors would like to thank Daniel Ungaretti and Rangel Baldasso for helpful discussions.
The research of B.N.B.L.\ was supported in part by CNPq grant 309468/2014-0 and FAPEMIG (Programa Pesquisador Mineiro). C.A. was supported by FAPESP, grant 2013/24928-2, and is thankful for the hospitality of the UFMG Mathematics Department.
The research of M.H.\ was partially supported by CNPq grant 406659/2016-1.

\end{document}